\documentclass[a4paper,12pt]{amsart}
\usepackage{amsmath}
\usepackage{amsmath}
\usepackage{amssymb}
\usepackage{amscd}
\usepackage{mathrsfs}
\usepackage[all]{xy}
\usepackage[dvipdfm]{graphicx}
\def\mathcal{\mathscr}
\everymath{\displaystyle}
\newtheorem{thm}{Theorem}[section]
\newtheorem{lem}[thm]{Lemma}
\newtheorem{cor}[thm]{Corollary}

\theoremstyle{definition}

\newtheorem{rem}[thm]{Remark}

\newtheorem{defn}[thm]{Definition}

\newcommand{\mca}[1]{{\mathcal{#1}}}

\def\Z{{\mathbb Z}}
\def\C{{\mathbb C}}

\def\R{{\mathbb R}}

\def\len{\text{\rm length}}

\def\Fix{\text{\rm Fix}}

\def\HZ{\text{\rm HZ}}
\def\id{\text{\rm id}}

\def\supp{\text{\rm supp}}

\begin{document}
\pagestyle{plain}
\thispagestyle{plain}

\title[Hofer-Zehnder capacity and a Hamiltonian circle action with noncontractible orbits]
{Hofer-Zehnder capacity and a Hamiltonian circle action with noncontractible orbits}

\author[Kei Irie]{Kei Irie}
\address{Department of Mathematics, Faculty of Science, Kyoto University,
Kyoto 606-8502, Japan}
\email{iriek@math.kyoto-u.ac.jp}

\date{\today}

\begin{abstract}
Let $(M,\omega)$ be an aspherical symplectic manifold, which is closed or convex. 
Let $U$ be an open set in $M$, which admits a circle action generated by an autonomous Hamiltonian $H \in C^\infty(U)$, such that 
each orbit of the circle action is not contractible in $M$. 
Under these assumptions, we prove that the Hofer-Zehnder capacity of $U$ 
is bounded by the Hofer norm of $H$. 
The proof uses a variant of the energy-capacity inequality, which is proved by the theory of action selectors. 
\end{abstract}

\maketitle

\section{Introduction}
\subsection{Hofer-Zehnder capacity} 
First we fix some notations. 
We set $S^1:=\R/\Z$. 
For any topological space $X$, we set 
$\pi_1'(X):=C^0(S^1,X)/\sim$, where $\gamma \sim \gamma'$ means that $\gamma$ and $\gamma'$ are homotopic. 
For each $\gamma \in C^0(S^1,X)$, $\bar{\gamma} \in C^0(S^1,X)$ is defined as $\bar{\gamma}(t):=\gamma(-t)$. 
Since $\gamma \sim \gamma' \implies \bar{\gamma} \sim \bar{\gamma'}$, one can define 
$\bar{\alpha} \in \pi_1'(X)$ for any $\alpha \in \pi_1'(X)$. 
When $X$ is path connected, $c_X$ denotes the element in $\pi_1'(X)$ which consists of contractible loops on $X$. 

We introduce a refinement of the Hofer-Zehnder capacity, taking into account free homotopy classes of periodic orbits. 
Let $(M,\omega)$ be a symplectic manifold. 
We always assume that $\partial M = \emptyset$. 
For $H \in C^\infty(M)$, its Hamiltonian vector field $X_H \in \mca{X}(M)$ is defined by the equation $\omega(X_H,\,\cdot\,)=-dH(\,\cdot \,)$.
For any $S \subset \pi'_1(M)$, $\mca{H}^S_{\HZ}(M,\omega)$ denotes the set of $H \in C_0^\infty(M)$ which 
satisfies the following properties: 
\begin{enumerate}
\item $H \le 0$ and $\{H=0\} \ne \emptyset$. 
\item There exists a nonempty open set $U \subset M$ such that $H|_U \equiv \min H$.
\item Any nonconstant periodic orbit $\gamma$ of $X_H$ satisfying $[\gamma] \in S$ has period $>1$. 
\end{enumerate}
Then we define 
\[
c^S_\HZ(M,\omega):= \sup \{ -\min H \mid H \in \mca{H}^S_\HZ(M,\omega) \}.
\]
Following properties are immediate from the definition:
\begin{itemize}
\item For any $S, S' \subset \pi_1'(M)$,  $S \subset S' \implies c^S_\HZ(M,\omega) \ge c^{S'}_\HZ(M,\omega)$. 
\item For any nonempty open set $U$ in $M$, let $i_U^M:U \to M$ denote the inclusion map, and let
$(i_U^M)_*: \pi_1'(U) \to \pi_1'(M)$ denote the induced map. Then, for any $S \subset \pi_1'(M)$, 
$c_\HZ^{(i_U^M)_*^{-1}(S)}(U,\omega) \le c_\HZ^S(M,\omega)$. 
\item Abbreviateing $c_\HZ(M,\omega):=c^{\pi'_1(M)}_\HZ(M,\omega)$, $c_{\HZ}(M,\omega) \le c_{\HZ}^S(M,\omega)$ for any $S \subset \pi_1'(M)$. 
Moreover, $c_\HZ(U,\omega) \le c_\HZ(M,\omega)$ for any nonempty open set $U$ in $M$. 
\end{itemize}
$c_\HZ(M,\omega)$ defined as above coincides with the original Hofer-Zehnder capacity (\cite{HZ}, \cite{HZ'}). 

\subsection{Main result}
First we fix some terminologies. 
Let $(M,\omega)$ be a symplectic manifold such that $\partial M=\emptyset$.
\begin{itemize}
\item $(M,\omega)$ is called \textit{aspherical} when $\omega|_{\pi_2(M)}=0$. 
\item $(M,\omega)$ is called \textit{convex} when there exists an increasing sequence
$M_1 \subset M_2 \subset \cdots$ of compact codimension $0$ submanifolds of $M$, which satisfies $\bigcup_i M_i=M$ and the following property
for each $i \ge 1$: there exists a vector field $X_i$ defined on some neighborhood of $\partial M_i$ in $M_i$, which
points strictly outwards on $\partial M_i$ and $L_{X_i}\omega=\omega$. 
\item
A circle action $S^1 \circlearrowleft M$ is called \textit{Hamiltonian} action generated by $H \in C^\infty(M)$, when there holds 
$\tfrac{d}{dt}(t \cdot x)|_{t=0} = X_H(x)$ for any $x \in M$. 
\end{itemize}

The main result of this note is the following:

\begin{thm}\label{thm:circle}
Let $(M,\omega)$ be a connected aspherical symplectic manifold, which is closed or convex. 
Let $U$ be an open set in $M$, which admits a Hamiltonian circle action generated by $H \in C^\infty(U)$. 
Suppose that for any $x \in U$, 
$\gamma^x: S^1 \to M; t \mapsto t \cdot x$ is not contractible in $M$, and 
$[\gamma^x] \in \pi_1'(M)$ does not depend on $x \in U$. 
Then, setting $\alpha:=[\gamma^x] \in \pi'_1(M)$, 
$c_\HZ^{(i^M_U)_*^{-1}(\{c_M, \bar{\alpha}\})}(U,\omega) \le \sup H - \inf H$. 
\end{thm}
\begin{rem}
In \cite{M}, L. Macarini gives a similar upper bound of $c_\HZ(U,\omega)$, provided that $U$ is a connected open set in a geometrically bounded symplectic manifold, 
and $U$ admits a free Hamiltonian circle action, which satisfies an additional condition on "the order of the action".
For precise statement, see Theorem 1.1 in \cite{M}. 
\end{rem}

Theorem \ref{thm:circle} is proved in section 2. First we give the following application:

\begin{cor}\label{cor:cotangent}
Let $N$ be a compact connected Riemannian manifold,
$\omega_N$ be the standard symplectic form on $T^*N$, and 
$DT^*N:=\{(q,p) \in T^*N \mid |p|<1\}$. 
Suppose that $N$ admits a circle action (which may not preserve the metric), such that for any $x \in N$,
$\gamma^x: S^1 \to N; t \mapsto t \cdot x$ is not contractible. 
Then, setting $\alpha:= [\gamma^x] \in \pi_1'(N)$, 
\[
c_\HZ^{\{c_N,\bar{\alpha}\}}(DT^*N, \omega_N) \le 2 \sup_{x \in N} \len(\gamma^x).
\]
\end{cor}
\begin{proof}
Let $Z$ be a vector field on $N$, which generates the given circle action, i.e. 
$Z_x:= \tfrac{d}{dt} (t \cdot x)|_{t=0}\,(x \in N)$. 
Up to reparamatrization of the action, we may assume that $\sup_{x \in N} |Z_x| \le \sup_{x \in N} \len(\gamma^x)=:L$. 
The circle action on $N$ naturally extends to a Hamiltonian circle action on $T^*N$ generated by 
$H \in C^\infty(T^*N)$, where $H(q,p):=p(Z_q)$. 
Since $DT^*N \subset H^{-1}\bigl((-L,L)\bigr)$, we get
\[
c_\HZ^{\{c_N,\bar{\alpha}\}}(DT^*N, \omega_N) \le c_\HZ^{\{c_N,\bar{\alpha}\}}(H^{-1}\bigl((-L,L)\bigr), \omega_N) \le 2L,
\]
where the second inequality follows from Theorem \ref{thm:circle}.
\end{proof}
\begin{rem}
In \cite{I}, the author proved $c_\HZ(DT^*N,\omega_N)<\infty$ under same assumption as Corollary \ref{cor:cotangent}, based on the correspondence between 
the pair-of-pants product in Floer homology of cotangent bundles and the loop product on homology of loop spaces. 
\end{rem}

As a specific case of Corollary \ref{cor:cotangent}, we recover the following result of M. Jiang \cite{J}:

\begin{cor}
Let $N$ be a flat torus: $N:=\R/a_1 \Z \times \cdots \times \R/a_n \Z$, where $n \ge 1$ and $0<a_1 \le \cdots \le a_n$. 
Then, $c_\HZ(DT^*N, \omega_N) \le 2a_1$. 
\end{cor} 

\section{Proof}

\subsection{Action selector} 
To prove Theorem \ref{thm:circle}, we use the notion of \textit{action selectors}. 
Let $(M, \omega)$ be an aspherical symplectic manifold, and 
$\mca{H}(M):= C_0^\infty(M \times [0,1])$. 

For any $H \in \mca{H}(M)$ and $t \in [0,1]$, $H_t \in C_0^\infty(M)$ is defined as 
$H_t(x):=H(t,x)$. 
For any $H \in \mca{H}(M)$, $(\varphi^t_H)_{0 \le t \le 1}$ is the flow generated by $(X_{H_t})_{0 \le t \le 1}$. i.e.
$\varphi^t_H: M \to M$ is defined as
\[
\varphi^0_H = \id_M, \qquad
\tfrac{d}{dt} \varphi^t_H = X_{H_t}(\varphi^t_H)\,(0 \le t \le 1).
\]

For any $H, K \in \mca{H}(M)$, we define $\bar{H}, H*K \in \mca{H}(M)$ as 
\[
\bar{H}(t,x):= -H(t, \varphi^t_H(x)), \qquad
H*K(t,x):= H(t,x) + K\bigl(t, (\varphi^t_H)^{-1}(x)\bigr). 
\]
It is easy to verifty the following properties: 
\begin{itemize}
\item $\varphi^t_{\bar{H}} = (\varphi^t_H)^{-1}$, $\varphi^t_{H*K} = \varphi^t_H \circ \varphi^t_K$ for any $0 \le t \le 1$. 
\item $(\mca{H}(M), *)$ is a group. The unit element is $0$, and the inverse of $H$ is $\bar{H}$. 
\end{itemize} 
 
For any $H \in \mca{H}(M)$ and $x \in \Fix(\varphi^1_H)$, we define $\gamma^x_H:S^1 \to M$ as $\gamma^x_H(t):=\varphi^t_H(x)$. We define
$\mca{P}(H):= \{ \gamma^x_H \mid x \in \Fix(\varphi^1_H) \}$.
$\mca{P}^\circ(H)$ denotes the set of $\gamma \in \mca{P}(H)$ which is contractible in $M$. 
Setting $D:=\{ z \in \C \mid |z| \le 1 \}$, 
for any contractible $\gamma:S^1 \to M$, we take $\bar{\gamma}: D \to M$ so that $\bar{\gamma}(e^{2\pi it})=\gamma(t)$ and define
\[
\mca{A}_H(\gamma):= \int_D \omega - \int_{S^1} H_t(\gamma(t)) dt. 
\]
It is well-defined since we have assumed that $(M,\omega)$ is aspherical. Then we define 
\[
\Sigma^\circ(H):= \{ \mca{A}_H(\gamma) \mid \gamma \in \mca{P}^\circ(H) \}.
\]
It is well-known that $\Sigma^\circ(H)$ is a nowhere dence subset in $\R$ (see Proposition 3.7 in \cite{S}). 
Finally, for any $H \in \mca{H}(M)$, we set 
\[
E_-(H):= - \int_0^1 \min H_t\,dt, \quad
E_+(H):= \int_0^1 \max H_t\,dt, \quad
\| H\|:= E_-(H) + E_+(H).
\]

\begin{defn}
Let $(M,\omega)$ be a connected aspherical symplectic manifold. 
An \textit{action selector} for $(M,\omega)$ is a map $\sigma: \mca{H}(M) \to \R$ which satisfies the following axioms: 
\begin{enumerate}
\item[(AS1)] $\sigma(H) \in \Sigma^\circ(H)$ for any $H \in \mca{H}(M)$. 
\item[(AS2)] For any $H \in \mca{H}^{\{c_M\}}_\HZ(M,\omega)$, $\sigma(H)=-\min H$. 
\item[(AS3)] $\sigma(H) \le E_-(H)$ for any $H \in \mca{H}(M)$. 
\item[(AS4)] $\sigma$ is continuous with respect to the $C^0$-topology of $\mca{H}(M)$. 
\item[(AS5)] $\sigma(H*K) \le \sigma(H) + E_-(K)$ for any $H, K \in \mca{H}(M)$.
\end{enumerate}
\end{defn}
\begin{rem}
The above set of axioms for action selectors follows that in \cite{FGS}, although our sign conventions are different from \cite{FGS}. 
Moreover, our notion of Hofer-Zehnder admissible Hamiltonians is wider than that in \cite{FGS}. 
\end{rem}

Our proof of Theorem \ref{thm:circle} is based on the following result:

\begin{thm}[\cite{S}, \cite{FS}]\label{thm:selector}
Let $(M,\omega)$ be a connected aspherical symplectic manifold. 
\begin{enumerate}
\item When $M$ is closed, there exists an action selector for $(M,\omega)$. 
\item When $M$ is convex, there exists an action selector for $(M,\omega)$. 
\end{enumerate}
\end{thm} 

(1) was proved by M. Schwarz in \cite{S}, based on the Piunkhin-Salamon-Schwarz isomorphism.  
(2) was proved by U. Frauenfelder and F. Schlenk in \cite{FS}, based on \cite{S} and 
Floer theory for convex symplectic manifolds (\cite{V}). 

\subsection{A variant of the energy-capacity inequality} 

First we prove the following result, which can be considered to be a variant of the energy-capacity inequality: 

\begin{thm}\label{thm:main} 
Let $(M,\omega)$ be a connected aspherical symplectic manifold, which is closed or convex. 
Let $U$ be an open set in $M$ and $H \in \mca{H}(M)$ such that: 
\begin{enumerate}
\item $\varphi^1_H|_U = \id_U$. 
\item For any $x \in U$, $\gamma^x_H: S^1 \to M; t \mapsto \varphi^t_H(x)$ is not contractible. 
Moreover, $[\gamma^x_H] \in \pi_1'(M)$ does not depend on $x \in U$. 
\end{enumerate}
Then, setting $\alpha:=[\gamma^x_H] \in \pi_1'(M)$, 
$c_\HZ^{(i^M_U)_*^{-1}(\{c_M,\bar{\alpha}\})}(U,\omega) \le \|H\|$. 
\end{thm} 

The proof is similar to the proof of the energy-capacity inequality in \cite{FGS} (section 2.1 in \cite{FGS}).
In the following, $\sigma: \mca{H}(M) \to \R$ denotes an action selector for $(M,\omega)$, which exists due to Theorem \ref{thm:selector}.

Suppose that $U, H, \alpha$ are as in Theorem \ref{thm:main}. 
We have to show $-\min K \le \|H\|$ for any $K \in \mca{H}^{\{c_M,\bar{\alpha}\}}_{\HZ}(M)$, $\supp K \subset U$.  
Since $K \in \mca{H}^{\{c_M\}}_\HZ(M)$, $\sigma(K)=-\min K$ by (AS2). 
Hence it is enough to show $\sigma(K) \le \|H\|$. 
First notice the following lemma:

\begin{lem}\label{lem:reparametrization}
For any $\chi \in C^\infty([0,1])$ satisfying $\int_0^1 \chi(t)\, dt=1$, 
we set $K_\chi \in \mca{H}(M)$ by $K_\chi(x,t):=K(x)\chi(t)$. 
Then, $\sigma(K_\chi) = \sigma(K)$. 
\end{lem}
\begin{proof}
For $0 \le s \le 1$, set $\chi_s:=s \chi + (1-s)$. 
Then, it is easy to verify that $\Sigma^\circ(K_{\chi_s})=\Sigma^\circ(K)$ for any $s$. 
By (AS1), $\sigma(K_{\chi_s}) \in \Sigma^\circ(K)$  for any $0 \le s \le 1$. On the otherhand, 
$\sigma(K_{\chi_s})$ depends continuously on $s$ by (AS4). 
Since $\Sigma^\circ(K)$ is nowhere dence, $[0,1]\to \R; s \mapsto \sigma(K_{\chi_s})$ is a constant function. Hence
$\sigma(K_\chi) =\sigma(K)$. 
\end{proof}
\begin{rem}
The above lemma is same as Lemma 2.2 in \cite{FGS}. We have included the proof for the convenience of the reader.
\end{rem}

Take $\chi \in C^\infty([0,1])$ so that $\int_0^1 \chi(t) \, dt=1$ and $\supp \chi \subset (1/2,1)$. 
By Lemma \ref{lem:reparametrization}, it is enough to show $\sigma(K_\chi) \le \|H\|$. 
After reparametrizing in $t$, we may assume that $H_t \equiv 0$ for $1/2 \le t \le 1$. 
Then 
\[
aK_\chi*H(t,x) = \begin{cases}
                       H(t,x) &(0 \le t \le 1/2) \\
                      aK_\chi(t,x) &(1/2 \le t \le 1)
                 \end{cases}.
\]
We claim that 
$\Sigma^\circ(aK_\chi*H) \subset \Sigma^\circ(H)$ for any $0 \le a \le 1$. 
Let $x \in \Fix(\varphi^1_{aK_\chi*H})$ such that $\gamma^x_{aK_\chi*H}$ is contractible in $M$. 
We distinguish two cases: 
\begin{itemize}
\item Suppose $x \in U$. 
Since $\varphi^1_H|_U=\id_U$, $x \in \Fix(\varphi^1_{aK_\chi})$. 
Since $0 \le a \le 1$ and $K \in \mca{H}^{\{\bar{\alpha}\}}_\HZ(M)$, $[\gamma^x_{aK_\chi}] \ne \bar{\alpha}$. 
On the otherhand, $\gamma^x_{aK_\chi*H}$ is a concatanation of 
$\gamma^x_H$ and $\gamma^x_{aK_\chi}$, and $[\gamma^x_H]=\alpha$. 
Hence $\gamma^x_{aK_\chi*H}$ is not contractible: it contradicts our assumption. 
\item Suppose $x \notin U$. 
Since $\varphi^1_{aK_\chi}|_{M \setminus U} = \id_{M \setminus U}$, $x \in \Fix(\varphi^1_H)$.
Using $\supp K \subset U$, it is easy to verify that 
$\gamma^x_{aK_\chi*H}=\gamma^x_H$, 
$aK_\chi*H(t,\gamma^x_{aK_\chi*H}(t)) = H(t,\gamma^x_H(t))$. 
Hence $\mca{A}_{aK_\chi*H}(\gamma^x_{aK_\chi*H}) = \mca{A}_H(\gamma^x_H) \in \Sigma^\circ(H)$.
\end{itemize}
Hence we have verified $\Sigma^\circ(aK_\chi*H) \subset \Sigma^\circ(H)$ for $0 \le a \le 1$. 
By (AS4), $\sigma(aK_\chi*H)$ depends continuously on $a$. 
Since $\Sigma^\circ(H)$ is nowhere dence, $[0,1] \to \R; a \mapsto \sigma(aK_\chi*H)$ is a constant function. Hence
$\sigma(H)= \sigma(K_\chi*H)$. 
Finally, we obtain $\sigma(K_\chi) \le \|H\|$ by 
\begin{align*}
\sigma(K_\chi)&= \sigma(K_\chi*H*\bar{H}) \le \sigma (K_\chi* H) + E_-(\bar{H}) \\
&= \sigma(H) + E_+(H) \le E_-(H)+E_+(H) = \| H\|.
\end{align*}
\qed

\subsection{Proof of Theorem \ref{thm:circle}} 

Finally we prove Theorem \ref{thm:circle}:

For any $S \subset \pi_1'(U)$, it is easy to verify that 
$c_\HZ^S(U,\omega) = \sup_V c_\HZ^{(i_V^U)_*^{-1}(S)}(V,\omega)$, 
where $V$ runs over all open sets of $U$ with compact closures.
Hence it is enough to show 
\[
c_\HZ^{(i_V^M)_*^{-1}(\{c_M,\bar{\alpha}\})}(V,\omega) \le \sup H - \inf H
\]
for any open $V \subset U$ with a compact closure. 
Since $H \in C^\infty(U)$ generates the given circle action on $U$, 
$V':=\bigcup_{0 \le t \le 1} \varphi^t_H(V)$ is invariant under the circle action, and 
it is again an open set of $U$ with compact closure. 
Then there exists $\rho \in C_0^\infty(U)$ such that $0 \le \rho \le 1$ and $\rho|_{V'} \equiv 1$. 
Then $\rho H \in C_0^\infty(U)$ extends to $M$, and $\varphi^t_{\rho H}|_V = \varphi^t_H|_V$ for any $0 \le t \le 1$. 
By adding constant to $H \in C^\infty(U)$ if necessary, we may assume that $\inf H \le 0 \le \sup H$. Then
$\inf \rho H \ge \inf H$, $\sup \rho H \le \sup H$, hence $\|\rho H\| \le \| H\|$. 
Finally we get 
\[
c_\HZ^{(i_V^M)_*^{-1}(\{c_M,\bar{\alpha}\})}(V,\omega) \le \|\rho H\| \le \| H \|.
\]
The first inequality follows from Theorem \ref{thm:main} (applied to $V$ and $\rho H$).
\qed

\textbf{Acknowledgement.} The author would like to appreciate Professor Kenji Fukaya for his encouregement.
The author is supported by Grant-in-Aid for JSPS fellows.

\end{document}